\documentclass[11pt]{amsart}
\usepackage{graphicx, enumerate, amsmath, amssymb,amsthm,manfnt,hyperref,amscd,braket}
\usepackage[utf8]{inputenc} 
\allowdisplaybreaks

\newcommand{\e}{\varepsilon}

\newcommand{\D}{\Omega}
\newcommand{\rl}{{\mathbb{R}}}
\newcommand{\cx}{{\mathbb{C}}}
\newcommand{\cn}{{\mathbb{C}^n}}

\newcommand{\Z}{\mathbb{Z}}

\newcommand{\dbar}{\overline{\partial}}

\theoremstyle{plain}
\newtheorem{theorem}[equation]{Theorem}
\newtheorem{proposition}[equation]{Proposition}
\newtheorem{lemma}[equation]{Lemma}
\newtheorem{corollary}[equation]{Corollary}

\theoremstyle{remark}
\newtheorem{remark}[equation]{Remark}
\numberwithin{equation}{section}

\newcommand{\Bp}{\mathbf{B}}
\newcommand{\Fo}{\mathbf{F}}

\newcommand{\s}{\mathbf{F}^2}

\newcommand{\Ht}{\mathbb{H}}

\newcommand{\Ud}{\mathbb{D}}

\title[Smoothing Properties of the Friedrichs Operator]{Smoothing Properties of the Friedrichs Operator on $L^p$ spaces}
\author{Liwei Chen}
\address[Liwei Chen]{The Ohio State University, Department of
Mathematics, Columbus, OH 43210}
\email{chen.1690@osu.edu}

\author{Yunus E. Zeytuncu}
\address[Yunus E. Zeytuncu]{University of Michigan - Dearborn, Department of Mathematics and Statistics, Dearborn, MI 48128}
\email{zeytuncu@umich.edu}

\thanks{The work of the second author is partially supported by a grant from the Simons Foundation (\#353525).}
\begin{document}

\maketitle
\begin{abstract}
We show that the Friedrichs operator exhibits smoothing properties in the $L^p$ scale. In particular we prove that on any smoothly bounded pseudoconvex domain the Friedrichs operator maps $A^2(\D)$ to $A^p(\D)$ for some $p>2$.
\end{abstract}

\section{Introduction}
Let $\Omega\subset \cn$ be a bounded domain. For $1\leq p<\infty$, $A^p(\Omega)=L^p(\Omega)\cap\mathcal{O}(\D)$ denotes the space of holomorphic functions on $\D$ that are $p$-integrable with respect to the Lebesgue measure on $\cn$. When $p=\infty$, we use the notation $H^{\infty}(\Omega)=L^{\infty}(\Omega)\cap\mathcal{O}(\D)$ to denote the space of bounded holomorphic functions on $\D$.

The Bergman projection operator $\Bp$ is the orthogonal projection operator from $L^2(\D)$ onto the Bergman space $A^2(\D)$. A closely related operator to $\Bp$ is the Friedrichs operator $\Fo$ on $A^2(\Omega)$ defined by
\begin{align*}
\mathbf{F}&: A^2(\Omega)\to A^2(\Omega)\\
\mathbf{F}&(g)=\mathbf{B}(\overline{g}).
\end{align*}
Originally introduced in \cite{Friedrichs}, $\Fo$ has been studied extensively on planar domains, see \cite{ShapiroUnbddQuad, ShapiroBook, PutinarShapiro1, PutinarShapiro2}. In particular, it is known that a planar domain is a quadrature domain if and only if the corresponding Friedrichs operator is of finite rank \cite[Theorem 2.4]{PutinarShapiro1}. Recently this operator has been studied on higher dimensional domains, and interesting smoothing properties have been obtained. It is known that the analytic properties of $\Bp$ depend on the geometry of the underlying domain $\D$. Under various geometric conditions on $\D$, $\Bp$ preserves function spaces such as $W^s(\D), L^p(\D)$, and $C^{\infty}(\overline{\D})$. However, since $\Bp$ reproduces holomorphic functions, it generally does not smooth the input function (see however \cite[Theorem 1.3]{HerMcN10} for a partial smoothing property). On the other hand, it is noticed in \cite{HerMcN10, HerMcNStr, RavZey} that $\Fo$ may demonstrate smoothing properties. Indeed, if $\D$ is smooth and satisfies condition R, then $\Fo$ maps any function in $A^2(\D)$ into 
$C^{\infty}(\overline{\D})$ \cite[Corollary 1.12]{HerMcNStr}. Furthermore, under some symmetry assumptions $\Fo$ holomorphically extends any input function to a strictly larger domain \cite[Theorem 1.1]{RavZey}.

In this note we study the smoothing properties of the Friedrichs operator in the $L^p$ scale. As in the Sobolev scale, we prove that $\Fo$ improves integrability. We start with a result similar to \cite[Theorem 1.1]{HerMcNStr} and illustrate how $\Fo$ gains $L^p$ integrability if $\Bp$ satisfies certain $L^p$ estimates.

\begin{theorem}
\label{lemFriedrichs1}
Let $\D\subset \cn$ be a smoothly bounded pseudoconvex domain and $p>2$. Suppose that there exists $r\in[p,\infty)$ such that $\Bp:L^r(\Omega)\to A^{p}(\Omega)$ is bounded, then $\Fo:A^2(\Omega)\to A^p(\Omega)$ is bounded. Moreover, for $g\in A^2(\Omega)$, we have
\[
\|\Fo(g)\|_{L^p}\le C\|g\|_{L^1}.
\]
\end{theorem}

Using this result, the Sobolev emdedding theorem, and the weighted results in \cite{BerChar} we also obtain low-level $L^p$ regularity of $\Fo$ on general domains. 
\begin{theorem}
\label{mainresult}
Let $\D\subset \cn$ be a smoothly bounded pseudoconvex domain. There exists $\e=\e(n,\Omega)>0$ such that $\Fo:A^2(\Omega)\to A^{p}(\Omega)$ is bounded for any $p\in[2,2+\e)$. Moreover, for $2\le p<2+\e$ and $g\in A^2(\Omega)$, we have
\begin{equation}
\|\Fo(g)\|_{L^p}\le C\|g\|_{L^1}.
\end{equation}
\end{theorem}
We note that a similar result for $\Bp$ is not known. It is noticed in some recent papers \cite{Zey13, Chen14, ChenZey, EdhMcN16, EdhMcN17} that on some non-smooth domains $\Bp$ exhibits degenerate $L^p$ regularity; however, it is not known if a similar result holds on smooth domains. Recently, in \cite{HarZey} low-level $L^p$ regularity is obtained for the $\dbar$-Neumann operator, and the canonical solutions operators, but not for the Bergman projection.  Furthermore, this low-level $L^p$ boundedness of $\Fo$ can be also turned into low-level Sobolev boundedness, with a gain compared to the similar statement for $\Bp$ in \cite{Kohn, BerChar}.

\begin{theorem}
\label{lowSob}
Let $\D\subset \cn$ be a smoothly bounded pseudoconvex domain, then $\Fo:A^2(\Omega)\to W^s(\Omega)$ is bounded for any $s<\eta/(4n)$, where $\eta$ is the Diedrich-Fornaess exponent of $\Omega$.
\end{theorem}

As indicated above, the Bergman projection $\Bp$ can exhibit various irregularity results on non-smooth domains. So it is a natural question whether $\Fo$ still has smoothing properties on non-smooth domains. In the third section of this note we study $\Fo$ on two model domains: Hartogs triangle
\[
\Ht=\{(z_1,z_2)\in\cx^2\,|\,|z_1|<|z_2|<1\}
\]
and an exponential version of it
\[
\Ht_{\infty}=\{(z_1,z_2)\in\cx^2\,|\,|z_1|<e^{-\frac{1}{|z_2|}},z_2\in\Ud^*\}.
\]
In particular we establish gain in integrability on these domains. 
\begin{proposition}
On the Hartogs triangle $\Ht$, the Friedrichs operator
\[
\Fo: A^2(\Ht)\to A^p(\Ht)
\]
is bounded for any $p\in[2,4)$, whereas the Bergman projection $\Bp$ does not map $L^2(\Ht)$ into $A^p(\Ht)$ for any $p>2$.
\end{proposition}

\begin{proposition}
On the exponential Hartogs triangle $\Ht_{\infty}$, the Friedrichs operator
\[
\Fo:A^2(\Ht_{\infty})\to A^p(\Ht_{\infty})
\]
is bounded for any $p\in[2,\infty)$, but not for $p=\infty$. 

\end{proposition}
Finally, within the same section, we look at the smoothing properties of $\Fo^2:=\Fo\circ\Fo$. It follows from the definition that $\Fo$ is not complex-linear but $\Fo^2$ is. Furthermore, $\Fo^2$ is positive and self-adjoint. The spectral properties of $\Fo^2$ can be used to relate two closed subspaces $A^2(\Omega)$ and $\overline{A^2(\Omega)}$ within $L^2(\Omega)$, see \cite[Section 8.4]{ShapiroBook}. In the context of this note, it turns out that the square of the Friedrichs operator gains more integrability than $\Fo$ itself on two model domains $\Ht$ and $\Ht_{\infty}$.


\section{Gaining Integrability and Differentiability on Smooth Pseudoconvex Domains}
In this section we prove Theorems \ref{lemFriedrichs1}, \ref{mainresult}, and \ref{lowSob}. The first result indicates how $\Fo$ gains integrability under certain assumption on $\Bp$, the second and third ones show that on any smoothly bounded pseudoconvex domain $\Fo$ exhibits regularity on $L^p$ and $W^s$ spaces.

\subsection{Gain in integrability}

We first consider a slightly more general version of Theorem \ref{lemFriedrichs1}.

\begin{theorem}
\label{lemFriedrichs2}
Let $\Omega\subset \cn$ be a smoothly bounded pseudoconvex domain. If for some $p>2$, there exists $r'\in[1,\infty]$ such that $\Bp:L^{p'}(\Omega)\to A^{r'}(\Omega)$ is bounded\footnote{When $r'=\infty$, the target space becomes $H^{\infty}(\D)$.}, where $p'$ is the conjugate exponent of $p$, then $\Fo:A^2(\Omega)\to A^p(\Omega)$ is bounded. Moreover, for $g\in A^2(\Omega)$, we have
\[
\|\Fo(g)\|_{L^p}\le C\|g\|_{L^1}.
\]
\end{theorem}

\begin{proof}
For $g\in A^2(\Omega)$, by duality of $L^p$ spaces, we have
\begin{equation*}
\|\Fo(g)\|_{L^p}=\sup_{\|h\|_{L^{p'}}=1}\left|\langle \Fo(g),h\rangle\right|=\sup_{\|h\|_{L^{p'}}=1}\left|\langle \bar g,\Bp(h)\rangle\right|=\sup_{\|h\|_{L^{p'}}=1}\left|\langle 1,g\Bp(h)\rangle\right|.
\end{equation*}
By Bell's lemma \cite[Lemma 2]{Bell81a}, there is bounded linear operator $\Phi^s:HW^s(\Omega)\to W_0^s(\Omega)$ such that $\Bp\Phi^s=Id$, where $HW^s(\Omega)=W^s(\Omega)\cap\mathcal{O}(\Omega)$. So we have
\[
\|\Fo(g)\|_{L^p}=\sup_{\|h\|_{L^{p'}}=1}\left|\langle \Bp\Phi^s(1),g\Bp(h)\rangle\right|
\]
for some $s$ to be determined later.

Since $\Omega$ is bounded smooth pseudoconvex, Catlin's result \cite[Theorem 3.1.4]{Catlin80} implies that $A^{\infty}(\Omega)=\mathcal{O}(\D)\cap C^{\infty}(\overline\D)$ is dense in $HW^s(\Omega)$. In particular, $A^{\infty}(\Omega)$ is dense in $A^2(\Omega)$. So the identity
\[
\langle \Bp\Phi^s(1),gf\rangle=\langle \Phi^s(1),\Bp(gf)\rangle=\langle \Phi^s(1),gf\rangle
\]
holds for all $f\in A^{\infty}(\Omega)$, and hence for all $f\in A^2(\Omega)$. Note that $L^{p'}(\Omega)\cap L^2(\Omega)$ is dense in $L^{p'}(\Omega)$. If $h\in L^2(\Omega)$, then $\Bp(h)\in A^2(\Omega)$. So we see that
\begin{align*}
\|\Fo(g)\|_{L^p}
&=\sup_{\|h\|_{L^{p'}}=1}\left|\langle \Bp\Phi^s(1),g\Bp(h)\rangle\right|\\
&=\sup_{\|h\|_{L^{p'}}=1,\, h\in L^{p'}(\Omega)\cap L^2(\Omega)}\left|\langle \Bp\Phi^s(1),g\Bp(h)\rangle\right|\\
&=\sup_{\|h\|_{L^{p'}}=1,\, h\in L^{p'}(\Omega)\cap L^2(\Omega)}\left|\langle \Phi^s(1),g\Bp(h)\rangle\right|\\
&=\sup_{\|h\|_{L^{p'}}=1}\left|\langle \Phi^s(1),g\Bp(h)\rangle\right|.
\end{align*}

By Sobolev embedding, we can pick $s$ sufficiently large, so that on $\Omega$ we have
\[
|\Phi^s(1)|\le M\cdot \delta^{2n}
\]
for some constant $M>0$, where $\delta$ is the distance function to the boundary $\partial\Omega$. Note that $\Bp(h)$ is holomorphic, by the mean-value theorem and H\"older inequality we have
\[
|\Bp(h)(z)|\le\frac{1}{|B(z;\delta(z))|}\int_{\Omega}|\Bp(h)|\le\frac{C|\Omega|^{1/r}}{\delta(z)^{2n}}\|\Bp(h)\|_{L^{r'}}\le\frac{C}{\delta(z)^{2n}}\|h\|_{L^{p'}},
\]
where $B(z;\delta(z))$ is the ball centered at $z$ of radius $\delta(z)$.

Combining all these estimates, we see that
\begin{align*}
\|\Fo(g)\|_{L^p}
&=\sup_{\|h\|_{L^{p'}}=1}\left|\langle \Phi^s(1),g\Bp(h)\rangle\right|\\
&\le\sup_{\|h\|_{L^{p'}}=1}\int_{\Omega}|\Phi^s(1)(z)g(z)\Bp(h)(z)|\,dV(z)\\
&\le\int_{\Omega}M\cdot\delta(z)^{2n}\cdot|g(z)|\cdot\frac{C}{\delta(z)^{2n}}\,dV(z)\\
&=C\|g\|_{L^1},
\end{align*}
which completes the proof.
\end{proof}

\begin{proof}[Proof of Theorem \ref{lemFriedrichs1}.]
If $\Bp:L^r(\Omega)\to A^{p}(\Omega)$ is bounded for some $r\in[p,\infty)$, then by duality we have for $f\in L^{p'}$
\[
\|\Bp(f)\|_{L^{r'}}=\sup_{\|h\|_{L^r}=1}|\langle\Bp(f),h\rangle|,
\]
where $r'$ and $p'$ are the conjugate exponents of $r$ and $p$ respectively. Since $\Bp$ is self-adjoint, we have
\[
\sup_{\|h\|_{L^r}=1}|\langle\Bp(f),h\rangle|=\sup_{\|h\|_{L^r}=1}|\langle f,\Bp(h)\rangle|.
\]
By H\"older inequality, we see that
\begin{align*}
\sup_{\|h\|_{L^r}=1}|\langle f,\Bp(h)\rangle|
&\le\sup_{\|h\|_{L^r}=1}\|f\|_{L^{p'}}\|\Bp(h)\|_{L^p}\\
&\le C\|f\|_{L^{p'}},
\end{align*}
which implies that $\Bp:L^{p'}(\Omega)\to A^{r'}(\Omega)$ is bounded. Now the conclusion follows from Theorem \ref{lemFriedrichs2}.
\end{proof}

Next we prove a partial low-level $L^p$ regularity result for the Bergman projection $\Bp$.

\begin{lemma}
\label{lemBergman}
Let $r>2$, $\Omega\subset \cn$ be a smoothly bounded pseudoconvex domain, and $\eta$ be the Diedrich-Fornaess exponent of $\Omega$.
\begin{enumerate}
\item If $r\ge\frac{2}{1-\eta}$, then $\Bp:L^r(\Omega)\to A^p(\Omega)$ is bounded for any $p\in[2,\frac{4n}{2n-\eta})$.
\item If $\frac{2}{1-\eta}>r>2$, then $\Bp:L^r(\Omega)\to A^p(\Omega)$ is bounded for any $p\in[2,\frac{4nr}{(2n-1)r+2})$.
\end{enumerate}
\end{lemma}

\begin{proof}
For $p\ge2$ and $f\in L^{r}(\Omega)$, by the Sobolev embedding theorem we have
\[
\|\Bp(f)\|_{L^{p}}\le\|\Bp(f)\|_{W^s},
\]
where $1/p=1/2-s/(2n)$ or equivalently $s=n(1-\frac{2}{p})$. $\Bp(f)$ is a holomorphic function, so for $0\le s<1/2$ we have
\[
\|\Bp(f)\|_{W^s}\le\|\Bp(f)\|_{L^2(\delta^{-2s})}
\]
by \cite[Theorem 4.2]{JerKen} and \cite[Lemma 1]{Det}. Indeed, these two norms are comparable to each other \cite{BerChar}. By the weighted estimate in \cite{BerChar} the paragraph above Theorem 2.4, we have
\[
\|\Bp(f)\|_{L^2(\delta^{-2s})}\le C\|f\|_{L^2(\delta^{-2s})}=C\left(\int_{\Omega}|f|^2\delta^{-2s}\right)^{1/2}
\]
for $2s<\eta$. Since $\eta<1$, two requirements $p\ge2$ and $2s<\eta$ indeed guarantee that $0\le s<1/2$.
By H\"older inequality, we see that
\[
\left(\int_{\Omega}|f|^2\delta^{-2s}\right)^{1/2}\le\left(\int_{\Omega}|f|^{r}\right)^{1/r}\left(\int_{\Omega}\delta^{-2s\cdot\frac{r}{r-2}}\right)^{\frac{r-2}{2r}}=C\|f\|_{L^{r}},
\]
provided $-2s\cdot\frac{r}{r-2}>-1$ or equivalently $2s<\frac{r-2}{r}$.

In summary, for $p\ge2$, $s=n(1-\frac{2}{p})$, and $f\in L^{r}(\Omega)$, if $2s<\eta$ and $2s<\frac{r-2}{r}$, then we have
\[
\|\Bp(f)\|_{L^{p}}\le C\|f\|_{L^{r}}.
\]
\begin{enumerate}
\item When $\frac{r-2}{r}\ge\eta$, i.e. $r\ge\frac{2}{1-\eta}$, the requirement $\eta>2s=2n(1-\frac{2}{p})$ gives $p<\frac{4n}{2n-\eta}$.
\item When $\frac{r-2}{r}<\eta$, i.e. $\frac{2}{1-\eta}>r$, the requirement $\frac{r-2}{r}>2s=2n(1-\frac{2}{p})$ gives $p<\frac{4nr}{(2n-1)r+2}$.
\end{enumerate}
\end{proof}

\begin{remark}

This result indicates that even though $\Bp$ may not preserve low-level $L^p$ spaces, it still doesn't degenerate as on non-smooth domains. In particular, if we choose $\e>0$ small enough and set $r=2+\e$, then the upper limit for the target space becomes $2+\e\frac{2}{4n+2n\e-\e}$. Therefore, for small enough $\e$ and $\e'<\e\left(\frac{2}{4n+2n\e-\e}\right)$ we show that $\Bp$ is bounded from $L^{2+\e}(\D)$ to $L^{2+\e'}(\D)$.
\end{remark}

With Lemma \ref{lemBergman} and Theorem \ref{lemFriedrichs1}, now we are ready to prove Theorem \ref{mainresult}. 

\begin{proof}[Proof of Theorem \ref{mainresult}]
By Lemma \ref{lemBergman}, we see that $\Bp:L^{r}(\Omega)\to A^{p}(\Omega)$ is bounded for any $p\in[2,\frac{4n}{2n-\eta})$, where $r\ge\frac{2}{1-\eta}$ and $\eta$ is the Diedrich-Fornaess exponent of $\Omega$. Now we apply Theorem \ref{lemFriedrichs1} and conclude that
\[
\Fo:A^2(\Omega)\to A^{p}(\Omega)
\]
 is bounded for any $p\in[2, 2+\e)$, where $\e=\frac{2\eta}{2n-\eta}$ and for $g\in A^2(\Omega)$
\[
\|\Fo(g)\|_{L^p}\le C\|g\|_{L^1}.
\]
\end{proof}

If condition R holds, \cite[Corollary 1.12]{HerMcNStr} shows that $\Fo:A^2(\Omega)\to C^{\infty}(\bar\Omega)$. Here we conclude an analogous statement, where condition R is replaced by an assumption on $L^p$ estimates.

\begin{corollary}
Let $\Omega\subset \cn$ be a smoothly bounded pseudoconvex domain. If for each $p\in[2,\infty)$, there exists $r\in(1,\infty)$ such that $\Bp:L^{r}(\Omega)\to A^p(\Omega)$ is bounded and the operator norm is uniform in $p$, then $\Fo:A^2(\Omega)\to H^{\infty}(\Omega)$.
\end{corollary}

\begin{proof}
By checking the proof of Theorem \ref{lemFriedrichs2} and the fact that $|\Omega|^{1/r}\le\max\{|\Omega|,1\}$, for $g\in A^2(\Omega)$ we have
\[
\|\Fo(g)\|_{L^p}\le C\|g\|_{L^1}
\]
for all $p\in[2,\infty)$ with $C>0$ independent of $p$. The conclusion follows by letting $p\to\infty$.
\end{proof}

\subsection{Gain in differentiability}

In this section we prove Theorem \ref{lowSob} as a consequence of Theorem \ref{mainresult} by showing the following lemma.

\begin{lemma}
\label{AptoWs}
Let $\Omega\subset \cn$ be a smoothly bounded pseudoconvex domain and let $\eta$ be the Diedrich-Fornaess exponent of $\Omega$. If $\Fo:A^2(\Omega)\to A^p(\Omega)$ is bounded for some $p<4n/(2n-\eta)$, then $\Fo:A^2(\Omega)\to W^s(\Omega)$ is bounded for any $s<(p-2)/2p$.
\end{lemma}

\begin{proof}
For any $g\in A^2(\Omega)$ and $0\le s<1/2$, the holomorphic function $\Fo(g)$ satisfies
\[
\|\Fo(g)\|_{W^s} \le \|\Fo(g)\|_{L^2(\delta^{-2s})}=\left(\int_{\Omega}|\Fo(g)|^2\delta^{-2s}\right)^{1/2}
\]
by \cite[Theorem 4.2]{JerKen} and \cite[Lemma 1]{Det}, where $\delta$ is the distance function to $\partial \Omega$. Then by H\"older inequality, we see that
\[
\left(\int_{\Omega}|\Fo(g)|^2\delta^{-2s}\right)^{1/2}\le\left(\int_{\Omega}|\Fo(g)|^p\right)^{1/p}\cdot\left(\int_{\Omega}\delta^{-2s\cdot p/(p-2)}\right)^{(p-2)/2p}=C\|\Fo(g)\|_{L^p}\le C\|g\|_{L^2},
\]
provided $p<4n/(2n-\eta)$ and $2sp/(p-2)<1$. This shows that 
\[
\|\Fo(g)\|_{W^s}\le C\|g\|_{L^2}
\]
for any $s<(p-2)/2p$.
\end{proof}

\begin{remark}
The converse is also true, for the larger range $s<\eta/2$. Indeed, by the Sobolev embedding theorem we have
\[
\Fo:A^2(\Omega)\to W^s(\Omega)\hookrightarrow L^p(\Omega),
\]
where $1/p=1/2-s/(2n)$. So if $s<\eta/2$, then $p<4n/(2n-\eta)$. 

\end{remark}

\begin{proof}[Proof of Theorem \ref{lowSob}]
By Lemma \ref{AptoWs}, Theorem \ref{mainresult} and the upper bound for $p$ in Lemma \ref{lemBergman} (1), it is easy to see that $\Fo:A^2(\Omega)\to W^s(\Omega)$ is bounded for any $s<\eta/(4n)$. 

\end{proof}

On planar domains, if the boundary is smooth enough then $\Fo$ is compact on the Bergman space \cite{Friedrichs}. Below by using Rellich's lemma, we conclude a similar statement in higher dimensions. See \cite{KraRoc} for stronger conclusions under condition R assumption.

\begin{corollary}
Let $\Omega$ be a bounded smooth pseudoconvex domain in $\cx^n$. Then the Friedrichs operator $\Fo:A^2(\Omega)\to A^2(\Omega)$ is compact.
\end{corollary}

\begin{proof}
This follows from Theorem \ref{lowSob} and the compactness of the inclusion $W^s(\Omega)\hookrightarrow L^2(\Omega)$.
\end{proof}


\section{The Friedrichs operator on $\Ht$ and $\Ht_{\infty}$}

The Bergman projection operator $\Bp$ exhibits interesting $L^p$ mapping properties on the Hartogs triangle
\[\Ht=\{(z_1,z_2)\in\cx^2\,|\,|z_1|<|z_2|<1\},\]
and on an exponential version of it
\[\Ht_{\infty}=\{(z_1,z_2)\in\cx^2\,|\,|z_1|<e^{-\frac{1}{|z_2|}},z_2\in\Ud^*\}.\]  
It turns out that $\Bp$ is bounded on $L^p(\Ht)$ if and only if $p\in(\frac{4}{3}, 4)$ \cite{ChaZey16, Chen14, EdhMcN16} and $\Bp$ is bounded on $L^p(\Ht_{\infty})$ if and only if $p=2$ \cite{ChenZey}. The main reason for these irregularities is the singular points on the boundaries of these domains. However, $\Fo$ and $\s$ still gain integrability on these domains.


\subsection{The Hartogs triangle $\Ht$}
\label{Hartogstriangle}
Since $\Bp$ is bounded on $L^p(\Ht)$ for $p\in(\frac{4}{3}, 4)$, $\Fo$ inherits this boundedness range. However, as on smooth pseudoconvex domains, $\Fo$ satisfies better estimates.

\begin{proposition}
On the Hartogs triangle $\Ht$, the Friedrichs operator
\[
\Fo: A^2(\Ht)\to A^p(\Ht)
\]
is bounded for any $p\in[2,4)$. 
\end{proposition}
We note that the Bergman projection $\Bp$ does not map $L^2(\Ht)$ into $A^p(\Ht)$ for any $p>2$. This can be easily seen by the example $f(z)=\frac{1}{(1-z_2)^{2/p}}$ for $z\in\Ht$ and given $p>2$.

\begin{proof}

We start working on the equivalent weighted space $A^2(\Ud^*,|z|^2)$, see  \cite{ChaZey16, Chen14}. The Bergman kernel is given by
\[
B(z,w)=\sum_{k\ge-1}\frac{k+2}{\pi}z^k\bar w^k
\]
for $z,w\in\Ud^*$. Given $f\in A^2(\Ud^*,|z|^2)$, we write
\[
f(w)=\sum_{k\ge-1}a_kw^k
\]
for $w\in\Ud^*$. A direct computation shows
\begin{equation}
\label{F(f)(z)}
\begin{split}
\Fo(f)(z)=\Bp(\bar f)(z)
&=\int_{\Ud^*}\Big(\sum_{k\ge-1}\frac{k+2}{\pi}z^k\bar w^k\Big)\Big(\sum_{n\ge-1}\bar a_n\bar w^n\Big)|w|^2\,dA(w)\\
&=\frac{\bar a_1}{2}z^{-1}+\bar a_0+\frac{3\bar a_{-1}}{2}z
\end{split}
\end{equation}
for $z\in\Ud^*$. When $2\le p<4$, we have
\begin{align*}
\left(\int_{\Ud^*}|\Fo(f)(z)|^p|z|^2\,dA(z)\right)^{1/p}
&\le\frac{|a_1|}{2}\left(\int_{\Ud^*}|z^{-1}|^p|z|^2\,dA(z)\right)^{1/p}+|a_0|\left(\int_{\Ud^*}|z|^2\,dA(z)\right)^{1/p}\\&+\frac{3}{2}|a_{-1}|\left(\int_{\Ud^*}|z|^p|z|^2\,dA(z)\right)^{1/p}=\\
&=\frac{|a_1|}{2}\Big(\frac{1}{4-p}\cdot 2\pi\Big)^{1/p}+|a_0|\Big(\frac{1}{4}\cdot 2\pi\Big)^{1/p}+\frac{3}{2}|a_{-1}|\Big(\frac{1}{p+4}\cdot 2\pi\Big)^{1/p}\\
&\le C_p(|a_1|+|a_0|+|a_{-1}|).
\end{align*}
Note that
\[
a_{-1}=\frac{1}{2\pi i}\int_{|z|=r}f(z)\,dz,
\]
so we have
\[
|a_{-1}|\le \frac{1}{2\pi}\cdot r\int_{0}^{2\pi}|f(re^{i\theta})|\,d\theta
\]
and hence
\[
|a_{-1}|\le\frac{3}{2\pi}\int_{\Ud^*}|f(z)||z|^2\,dA(z).
\]
Similarly, we have
\begin{equation}
\label{|a_0|}
|a_0|\le\frac{2}{\pi}\int_{\Ud^*}|f(z)||z|^2\,dA(z)
\end{equation}
and 
\begin{equation}
\label{|a_1|}
|a_1|\le \frac{5}{2\pi}\int_{\Ud^*}|f(z)||z|^2\,dA(z).
\end{equation}
Therefore, we have a better estimate
\[
\left(\int_{\Ud^*}|\Fo(f)(z)|^p|z|^2\,dA(z)\right)^{1/p}\le C_p\int_{\Ud^*}|f(z)||z|^2\,dA(z)
\]
for $p\in[2,4)$, which proves the proposition.
\end{proof}

\subsection{The exponential Hartogs triangle $\Ht_{\infty}$}
\label{expHartogs}

We know that the Bergman projection
\[
\Bp:L^p(\Ht_{\infty})\to A^p(\Ht_{\infty})
\]
is bounded only when $p=2$, see \cite[Theorem 1.3]{ChenZey}. Now we turn to $\Fo$.

We first look at the auxiliary function\footnote{We write $A\approx B$ to mean that there exists $c>0$ such that $c^{-1}A<B<cA$.}
\[
\tilde I(x)=\int_0^1r^{x}e^{-\frac{2}{r}}\,dr
\]
for $x\in\rl$. By the asymptotic behavior of $I(x)$ \cite[Lemma 2.1]{ChenZey}, we see that
\begin{equation}
\label{xtoinfty}
\tilde I(x)\approx\frac{1}{x+1}
\end{equation}
and
\begin{equation}
\label{-xtoinfty}
\tilde I(-x)\approx\left(\frac{1}{2}\right)^{x-1}\Gamma(x-1)
\end{equation}
as $x\to\infty$. Let $\alpha=(\alpha_1,\alpha_2)\in\Z^2$, we consider the orthonormal basis $\{c_{\alpha}z_1^{\alpha_1}z_2^{\alpha_2}\}_{\alpha_1\ge0,\alpha_2\in\Z}$ for $A^2(\Ht_{\infty})$, where
\begin{equation}
\label{calpha}
c_{\alpha}^{-2}=\int_{\Ht_{\infty}}|z_1|^{2\alpha_1}|z_2|^{2\alpha_2}\,dV(z)=\frac{\pi}{\alpha_1+1}\int_{\Ud^*}|z_2|^{2\alpha_2}e^{-\frac{2\alpha_1+2}{|z_2|}}\,dA(z_2).
\end{equation}

\begin{proposition}
\label{A^p-regexp}
On the exponential Hartogs triangle $\Ht_{\infty}$, the Friedrichs operator
\[
\Fo:A^2(\Ht_{\infty})\to A^p(\Ht_{\infty})
\]
is bounded for $p\in[2,\infty)$, but not for $p=\infty$. 

\end{proposition}

\begin{proof}
A straightforward computation shows
\[
\Fo(z_2)=\frac{c}{z_2}\notin H^{\infty}(\Ht_{\infty})
\]
for some constant $c$ (see \eqref{F(f)(z)exp} below). So $\Fo$ does not map $H^{\infty}(\Ht_{\infty})$ into $H^{\infty}(\Ht_{\infty})$.

For the boundedness part, let $z,w\in\Ht_{\infty}$, then we have the Bergman kernel
\[
B(z,w)=\sum_{\alpha_1\ge0,\alpha_2\in\Z}c_{\alpha}^2z_1^{\alpha_1}z_2^{\alpha_2}\bar w_1^{\alpha_1}\bar w_2^{\alpha_2}.
\]
For $f\in A^2(\Ht_{\infty})$, we write
\[
f(z)=\sum_{\alpha_1\ge0,\alpha_2\in\Z}a_{\alpha}z_1^{\alpha_1}z_2^{\alpha_2}.
\]
Then we see that
\begin{equation}
\label{F(f)(z)exp}
\begin{split}
\Bp(\bar f)(z)
&=\int_{\Ht_{\infty}}\Big(\sum_{\alpha_1\ge0,\alpha_2\in\Z}c_{\alpha}^2z_1^{\alpha_1}z_2^{\alpha_2}\bar w_1^{\alpha_1}\bar w_2^{\alpha_2}\Big)\Big(\overline{\sum_{\beta_1\ge0,\beta_2\in\Z}a_{\beta}w_1^{\beta_1}w_2^{\beta_2}}\Big)\,dV(w)\\
&=\pi\sum_{\alpha_2,\beta_2\in\Z}\int_{\Ud^*}c_{(0,\alpha_2)}^2z_2^{\alpha_2}\bar a_{(0,\beta_2)}\bar w_2^{\alpha_2+\beta_2}e^{-\frac{2}{|w_2|}}\,dA(w_2)\\
&=C\sum_{\alpha_2\in\Z}c_{(0,\alpha_2)}^2z_2^{\alpha_2}\bar a_{(0,-\alpha_2)}
\end{split}
\end{equation}
for $z\in\Ht_{\infty}$. Note that the residue of a meromorphic function can be computed by integration along a circle centered at $0$ of some radius $r$. So for $z_2\in\Ud^*$ and $r_1<\exp(-1/|z_2|)$ we have
\[
\int_{|z_1|=r_1}\frac{f(z_1,z_2)\,dz_1}{z_1}=2\pi i\sum_{\alpha_2\in\Z}a_{(0,\alpha_2)}z_2^{\alpha_2},
\]
or equivalently
\[
\frac{1}{2\pi}\int_0^{2\pi}f(r_1e^{i\theta_1},z_2)\,d\theta_1=\sum_{\alpha_2\in\Z}a_{(0,\alpha_2)}z_2^{\alpha_2}.
\]
Similarly, for $r_2\in(0,1)$ we have
\[
\frac{1}{2\pi}\int_{|z_2|=r_2}\int_0^{2\pi}\frac{f(r_1e^{i\theta_1},z_2)\,d\theta_1}{z_2^{\alpha_2+1}}\,dz_2=2\pi ia_{(0,\alpha_2)},
\]
or equivalently
\[
\frac{1}{4\pi^2}\int_0^{2\pi}\int_0^{2\pi}\frac{f(r_1e^{i\theta_1},r_2e^{i\theta_2})\,d\theta_1\,d\theta_2}{(r_2e^{i\theta_2})^{\alpha_2}}=a_{(0,\alpha_2)}.
\]
So for $\alpha_2\in\Z$ we obtain
\[
a_{(0,-\alpha_2)}=\frac{1}{4\pi^2}\int_0^{2\pi}\int_0^{2\pi}f(r_1e^{i\theta_1},r_2e^{i\theta_2})r_2^{\alpha_2}e^{i\alpha_2\theta_2}\,d\theta_1\,d\theta_2,
\]
and hence
\[
|a_{(0,-\alpha_2)}|\le\frac{1}{4\pi^2}\int_0^{2\pi}\int_0^{2\pi}|f(r_1e^{i\theta_1},r_2e^{i\theta_2})|r_2^{\alpha_2}\,d\theta_1\,d\theta_2.
\]
Therefore, we have
\[
\int_0^1\int_0^{e^{-\frac{1}{r_2}}}|a_{(0,-\alpha_2)}|r_1r_2^{-\alpha_2+1}\,dr_1\,dr_2\le\frac{1}{4\pi^2}\int_0^1\int_0^{e^{-\frac{1}{r_2}}}\int_0^{2\pi}\int_0^{2\pi}|f(r_1e^{i\theta_1},r_2e^{i\theta_2})|r_1r_2\,d\theta_1\,d\theta_2\,dr_1\,dr_2,
\]
which implies
\begin{equation}
\label{|a_alpha|}
|a_{(0,-\alpha_2)}|\le C\left(\tilde I(-\alpha_2+1)\right)^{-1}\int_{\Ht_{\infty}}|f(z)|\,dV(z).
\end{equation}
For $\alpha_2\in\Z$, if we let $\alpha_1=0$ in \eqref{calpha}, then we obtain
\begin{equation}
\label{c^2_alpha}
c_{(0,\alpha_2)}^2=C\left(\tilde I(2\alpha_2+1)\right)^{-1}.
\end{equation}
Note that for each $\alpha_2\in\Z$, we have
\[
\|z_2^{\alpha_2}\|_{L^p(\Ht_{\infty})}=\left(\int_{\Ht_{\infty}}|z_2|^{\alpha_2p}\,dV(z)\right)^{1/p}=C_p\left(\tilde I(\alpha_2p+1)\right)^{1/p}.
\]
Therefore, we have the following estimate
\begin{equation*}
\left(\int_{\Ht_{\infty}}|\Bp(\bar f)(z)|^p\,dV(z)\right)^{1/p}\le C_p\sum_{\alpha_2\in\Z}\left(\tilde I(2\alpha_2+1)\right)^{-1}\left(\tilde I(\alpha_2p+1)\right)^{1/p}\left(\tilde I(-\alpha_2+1)\right)^{-1}\int_{\Ht_{\infty}}|f(z)|\,dV(z).
\end{equation*}

Note that by the asymptotic behavior of $\tilde I$ and Stirling's formula, we have
\begin{align*}
\sum_{\alpha_2\ge0}\left(\tilde I(2\alpha_2+1)\right)^{-1}\left(\tilde I(\alpha_2p+1)\right)^{1/p}\left(\tilde I(-\alpha_2+1)\right)^{-1}
&\approx\sum_{\alpha_2\ge0}\frac{2\alpha_2+2}{(\alpha_2p+2)^{1/p}\cdot\big(\frac{1}{2}\big)^{\alpha_2-2}\Gamma(\alpha_2-2)}\\
&\lesssim\sum_{\alpha_2\ge0}\frac{2^{\alpha_2}}{\alpha_2!}<\infty.
\end{align*}
and for $k=-\alpha_2$
\begin{align*}
\sum_{\alpha_2<0}\left(\tilde I(2\alpha_2+1)\right)^{-1}\left(\tilde I(\alpha_2p+1)\right)^{1/p}\left(\tilde I(-\alpha_2+1)\right)^{-1}
&=\sum_{k>0}\Big(\tilde I(-2k+1)\Big)^{-1}\Big(\tilde I(-kp+1)\Big)^{\frac{1}{p}}\Big(\tilde I(k+1)\Big)^{-1}\\
&\approx\sum_{k>0}\frac{(k+2)\big(\frac{1}{2}\big)^{\frac{kp-2}{p}}\big(\Gamma(kp-2)\big)^{1/p}}{\big(\frac{1}{2}\big)^{2k-2}\Gamma(2k-2)}\\
&\approx\sum_{k>0}\frac{(k+2)2^k(kp-3)^{\frac{1}{2p}}\Big(\frac{kp-3}{e}\Big)^{k-3/p}}{(2k-3)^{1/2}\Big(\frac{2k-3}{e}\Big)^{2k-3}}\\
&\lesssim\sum_{k>0}\frac{\tau^k}{k^k}<\infty
\end{align*}
for some constant $\tau>1$. So we have
\[
\sum_{\alpha_2\in\Z}\left(\tilde I(2\alpha_2+1)\right)^{-1}\left(\tilde I(\alpha_2p+1)\right)^{1/p}\left(\tilde I(-\alpha_2+1)\right)^{-1}<\infty.
\]
and hence a better estimate
\[
\left(\int_{\Ht_{\infty}}|\Bp(\bar f)(z)|^p\,dV(z)\right)^{1/p}\le C_p\int_{\Ht_{\infty}}|f(z)|\,dV(z),
\]
which implies that
\[
\Fo:A^2(\Ht_{\infty})\to A^p(\Ht_{\infty})
\]
is bounded for $p\in[2,\infty)$.
\end{proof}

\subsection{The square of the Friedrichs operator, $\s=\Fo\circ\Fo$}

We have already seen in \S \ref{Hartogstriangle} that 
\begin{enumerate}
\item $\Bp:L^p(\Ht)\to A^p(\Ht)$ is bounded for $p\in(4/3,4)$, and
\item $\Fo:A^2(\Ht)\to A^p(\Ht)$ is bounded for $p\in[2,4)$.
\end{enumerate}

We again consider the orthonormal basis $\Big\{\sqrt{\frac{n+2}{\pi}}z^n\Big\}_{n\ge-1}$ for the space $A^2(\Ud^*,|z|^2)$. So we have the formula for the Bergman kernel on $\Ud^*\times\Ud^*$
\[
B(z,w)=\sum_{n\ge-1}\frac{n+2}{\pi}z^n\bar w^n=\frac{1}{\pi}\cdot\frac{1}{z\bar w}\cdot\frac{1}{(1-z\bar w)^2}.
\]

Note that when $p=4$, $\Bp$ and $\Fo$ are unbounded since $\Bp(\bar z)=c/z\notin A^4(\Ud^*,|z|^2)$ for some constant $c$. However, $\s(z)=\frac{3}{4}z\in A^4(\Ud^*,|z|^2)$.

\begin{proposition}
On the Hartogs triangle $\Ht$, the operator $\s$ is bounded on $L^p(\Ht)$ for all $p\in[2,\infty]$. Indeed, we have the mapping property that
\begin{enumerate}
\item $\s:A^2(\Ht)\to A^p(\Ht)$ is bounded for all $p\in[2,4)$;
\item $\s:A^4(\Ht)\to A^p(\Ht)$ is bounded for all $p\in[4,\infty]$.\footnote{When $p=\infty$, the target space becomes $H^{\infty}(\Ht)$.}
\end{enumerate}
In either case, we have a better estimate $\|\s(f)\|_{L^p}\le C_p\|f\|_{L^1}$.
\end{proposition}

\begin{proof}
Case (1) follows from the fact that $\Fo:A^2\to A^p$ is bounded for all $p\in[2,4)$ and that $\s=\Fo\circ\Fo$. Note that $\s(1/z)=c/z\notin A^4(\Ud^*,|z|^2)$ for some constant $c$. So $\s$ does not map $A^2(\Ud^*,|z|^2)$ into $A^4(\Ud^*,|z|^2)$.

For case (2), we assume $p\ge4$, and thus $1/z\notin A^p(\Ud^*,|z|^2)$. For any $f\in A^p(\Ud^*,|z|^2)$, since $A^p(\Ud^*,|z|^2)\subset A^2(\Ud^*,|z|^2)$, we see that $f$ has the expansion
\[
f(z)=\sum_{n=0}^{\infty}a_nz^n.
\]
Note that $1/z\notin A^p(\Ud^*,|z|^2)$, $f$ does not have the term $a_{-1}z^{-1}$. Therefore, by \eqref{F(f)(z)} we have
\begin{align*}
\s(f)(z)
&=\Fo\circ\Fo(f)(z)\\
&=\Fo\Big(\frac{\bar a_1}{2}z^{-1}+\bar a_0\Big)(z)\\
&=a_0+\frac{3}{4}a_1z.
\end{align*}

By \eqref{|a_0|} and \eqref{|a_1|}, we have
\begin{equation*}
|a_0|\le\frac{2}{\pi}\int_{\Ud^*}|f(z)||z|^2\,dA(z)
\end{equation*}
and 
\begin{equation*}
|a_1|\le \frac{5}{2\pi}\int_{\Ud^*}|f(z)||z|^2\,dA(z).
\end{equation*}

So we obtain for $p\in[4,\infty)$
\[
\left(\int_{\Ud^*}|\s(f)(z)|^p|z|^2\,dA(z)\right)^{1/p}\le C_p(|a_0|+|a_1|)\le C_p\int_{\Ud^*}|f(z)||z|^2\,dA(z),
\]
and for $p=\infty$
\[
\left\|\s(f)(z)\right\|_{\infty}\le C(|a_0|+|a_1|)\le C\int_{\Ud^*}|f(z)||z|^2\,dA(z).
\]
\end{proof}

\begin{remark}
We note that the idea here is different from using a Schur's test argument, since we only use the holomorphicity of $f\in A^2(\Ht)$ rather than looking at some kernel estimates.
\end{remark}

We see in the proof of Proposition \ref{A^p-regexp} that $\Fo$ does not map $H^{\infty}(\Ht_{\infty})$ into $H^{\infty}(\Ht_{\infty})$. It is easy to see that $\s$ does not map $A^2(\Ht_{\infty})$ into $H^{\infty}(\Ht_{\infty})$, since by \eqref{F(f)(z)exp} we have
\[
\s(\frac{1}{z_2})=\frac{c}{z_2}\notin H^{\infty}(\Ht_{\infty})
\]
for some constant $c$. However, we have the following.

\begin{proposition}
On the exponential Hartogs triangle $\Ht_{\infty}$, the operator $\s:H^{\infty}(\Ht_{\infty})\to H^{\infty}(\Ht_{\infty})$ is bounded.
\end{proposition}

\begin{proof}
For $f\in H^{\infty}(\Ht_{\infty})$, we write
\[
f(z)=\sum_{\alpha_1,\alpha_2\ge0}a_{\alpha}z_1^{\alpha_1}z_2^{\alpha_2}.
\]
By \eqref{F(f)(z)exp}, we have
\begin{align*}
\s(f)(z)
&=\Fo\circ\Fo(f)(z)\\
&=\Fo\Big(C\sum_{\alpha_2\le0}c^2_{(0,\alpha_2)}z_2^{\alpha_2}\bar a_{(0,-\alpha_2)}\Big)(z)\\
&=C\sum_{\alpha_2\ge0}c^2_{(0,\alpha_2)}c^2_{(0,-\alpha_2)}a_{(0,\alpha_2)}z_2^{\alpha_2}
\end{align*}
for some constant $C$. Note that by \eqref{|a_alpha|} and \eqref{c^2_alpha} we have
\begin{equation*}
|a_{(0,\alpha_2)}|\le C\left(\tilde I(\alpha_2+1)\right)^{-1}\int_{\Ht_{\infty}}|f(z)|\,dV(z)\le C\left(\tilde I(\alpha_2+1)\right)^{-1}\|f\|_{L^{\infty}}
\end{equation*}
and
\begin{equation*}
c_{(0,\alpha_2)}^2=C\left(\tilde I(2\alpha_2+1)\right)^{-1}.
\end{equation*}
Hence for $z\in\Ht_{\infty}$, we see that
\begin{align*}
|\s(f)(z)|
&\le C\sum_{\alpha_2\ge0}c^2_{(0,\alpha_2)}c^2_{(0,-\alpha_2)}|a_{(0,\alpha_2)}|\\
&\le C\sum_{\alpha_2\ge0}\Big(\tilde I(2\alpha_2+1)\Big)^{-1}\Big(\tilde I(-2\alpha_2+1)\Big)^{-1}\left(\tilde I(\alpha_2+1)\right)^{-1}\|f\|_{L^{\infty}}.
\end{align*}
Note that for $k=\alpha_2\ge0$, we have
\[
\sum_{k\ge0}\Big(\tilde I(2k+1)\Big)^{-1}\Big(\tilde I(-2k+1)\Big)^{-1}\left(\tilde I(k+1)\right)^{-1}\approx\sum_{k\ge0}\frac{(k+2)(2k+2)}{\big(\frac{1}{2}\big)^{2k-2}\Gamma(2k-2)}<\infty.
\]
Therefore, we obtain
\[
\|\s(f)\|_{L^{\infty}}\le C\|f\|_{L^{\infty}}.
\]
\end{proof}

\begin{remark}
In both cases of $\Ht$ and $\Ht_{\infty}$ we observe different mapping properties for $\s$, $\Fo$, and $\Bp$, where $\s$ smooths the most. It is a curious question to investigate similar smoothing properties and to compare the gain in between these operators on general pseudoconvex domains.
\end{remark}


\end{document}